\documentclass{amsart}
\usepackage{graphicx}
\usepackage{array}
\usepackage{epsfig}
\usepackage{epstopdf}
\usepackage{amsmath, amsfonts,amssymb,amsthm}
\usepackage{subfigure}
\theoremstyle{plain}
\newtheorem{lem}{Lemma}[section]
\newtheorem{thm}[lem]{Theorem}
\newtheorem{prop}[lem]{Proposition}
\newtheorem{cor}[lem]{Corollary}

\theoremstyle{definition}
\newtheorem{defn}{Definition}[section]

\newtheorem{exmp}{Example}[section]

\begin{document}

\title{On Brylawski's generalized duality}

\author{Gary Gordon}
\address{Dept. of Mathematics\\
	Lafayette College\\
   	Easton, PA  18042-1781}
 \email{gordong@lafayette.edu}

\maketitle

\begin{abstract}  We introduce a notion of duality (due to Brylawski) that generalizes matroid duality to arbitrary rank functions.  This generalized duality allows for  generalized operations (deletion and contraction) and a generalized polynomial based on the matroid Tutte polynomial.  This polynomial satisfies a  deletion-contraction recursion.  We explore this notion of duality for greedoids, antimatroids and demi-matroids, proving that matroids correspond precisely to objects that are simultaneously greedoids and ``dual'' greedoids.
\end{abstract}

\maketitle
\section{Introduction}  \label{S:intro}
Duality plays a central role in graph theory and matroid theory.  While only planar graphs have graphic duals, all matroids have duals.  Since graphs are matroids, and since geometric duality for planar graphs coincides with matroid duality when the graph is planar, we can view matroid duality as a way to define duals for non-planar graphs.

A matroid can be described by its rank function, and this rank function can then be used to define the three operations of deletion, contraction and duality.  Brylawski \cite{br} realized it is possible to extend all three of these operations  to arbitrary ``rank'' functions.  If $r:2^S \to \mathbb{Z}$ is any  function, then one can define a dual structure via a dual rank function $r^*$:
$$r^*(A)=|A|+r(S-A)-r(S).$$
Then deletion can be defined as a restriction of the rank function $r$:
$$r_{G-p}=r|_{G-p},$$ and contraction can be defined using deletion and duality:
$$G/p=(G^*-p)^*.$$  

This extends the idea of duality to non-matroidal structures. These definitions (along with some background on matroids) and basic results are given in Section~\ref{S:def}.  While this generalized duality is difficult to interpret combinatorially (in particular, we can have $r^*(A)<0$ for a subset $A$), we can prove several generalizations of well-known formulas involving deletion, contraction and duality. 

Others have worked on characterizing matroid duality through its  properties.  Kung~\cite{k} shows that matroid duality is the {\it only} involution on the class of matroids that interchanges deletion and contraction:  $(G-p)^*=G^*/p$ and $(G/p)^*=G^*-p.$  Kung's approach is generalized by Whittle~\cite{wh}, who extends duality to {\it $k$-polymatroids}, where the dual rank function satisfies $f^*(A)=k|A|+f(S-A)-f(S)$.  This also interchanges deletion and contraction, and has the involution property $f^{**}=f$.  Bland and Dietrich~\cite{bd} also investigate duality,  concentrating on involutions, but also considering the class of oriented matroids..

Our primary motivation in this work is the close connection to the Tutte polynomial, the subject of Section~\ref{S:tutte}.   When $M$ is a matroid, the Tutte polynomial $t(M;x,y)$ is generally defined in one of two equivalent ways:   via a subset expansion, or inductively, through a deletion-contraction recursion.  The subset expansion uses the rank function:
$$t(M;x;y)=\sum_{A \subseteq S}(x-1)^{r(S)-r(A)}(y-1)^{|A|-r(A)},$$
and this will allow us to define a ``Tutte polynomial'' in this general setting.

The matroid version of the deletion-contraction recursion is the following  formula: $$t(M;x,y)= t(M-p;x,y)+t(M/p;x,y).$$   This formula is generalized in Theorem~\ref{T:tutte-gen}(1) in Section~\ref{S:tutte}:
$$f(G;t,z)=t^{r(G)-r(G-p)}f(G-p;t,z)+z^{1-r(p)}f(G/p;t,z).$$
  Theorem~\ref{T:tutte-gen}(1)   also  generalizes the deletion-contraction formula for the Tutte polynomial of a greedoid (Proposition 2.5 of \cite{gm}).   Theorem~\ref{T:tutte-gen}(2) shows this general Tutte polynomial is well-behaved with respect to  generalized duality (assuming $r(\emptyset)=0$): $$f(G^*;t,z)=f(G;z,t).$$

{\it Greedoids} are generalizations of matroids, and there are many interesting combinatorial structures that have meaningful interpretations as greedoids, but not matroids.  (For instance, trees form greedoids, but the matroid associated with a tree is trivial.)  In Section~\ref{S:greedoid}, we examine our generalized duality for greedoids.  The main results are a characterization of the rank axioms dual greedoids satisfy (Theorem~\ref{T:dual-greedoid1}) and a result that shows $\mathcal{G}\cap \mathcal{G}^*=\mathcal{M}$, where $\mathcal{G}$ is the class of all greedoids, $\mathcal{G}^*$ is the class of all greedoid {\it duals}, and $\mathcal{M}$ is the class of all matroids (Theorem~\ref{T:dual-greedoid2}).

In Section~\ref{S:matroid-gen}, we conclude by considering applications to {\it antimatroids} (a well-studied class of greedoids) and {\it demi-matroids}, another matroid generalization  introduced recently in \cite{bjms}.  For antimatroids, we interpret the dual rank combinatorially in terms of {\it convex closure}  (Theorem~\ref{T:anti-rank}).  For demi-matroids, we examine the connection between our generalized duality and these objects, characterizing precisely the properties the rank function $r$ must satisfy to produce a demi-matroid (Theorem~\ref{T:demi}).

Finally, I offer my  gratitude to Tom Brylawski (1944--2007) for many fruitful discussions on this topic.  This approach to duality is due to him, and many of the results given were originally proven by him.  His influence on this author goes well beyond the present work, and this 
paper is dedicated to his memory.  A memorial volume  of the {\it European Journal of Combinatorics} includes a tribute to Tom and his work  \cite{thb}.

\section{Definitions}\label{S:def}
\subsection{Matroids via the rank function}
There are many {\it cryptomorphically} equivalent ways to define a matroid. For instance, among other formulations, matroids can be defined via independent sets, bases, circuits, or flats.  In this paper, we use the rank function.

\begin{defn} A \textit{matroid} $M$ is a pair $(S,r)$ where $S$ is a finite set and $r:2^S \to \mathbb{Z}^+\cup\{0\}$ such that:
\begin{enumerate}
\item [(R0)] $r(\emptyset)=0$ [\textit{normalization}]
\item [(R1)] $r(A) \leq r(A \cup p) \leq r(A)+1$ [\textit{unit rank increase}]
\item [(R2)] $r(A\cap B)+r(A \cup B) \leq r(A)+r(B)$ [\textit{semimodularity}]
\end{enumerate}

\end{defn}
$S$ is the {\it ground set} of the matroid.   Assuming (R0) and (R1), we can replace (R2) with

\textit{Local semimodularity:}

\begin{enumerate}
\item [(R2$'$)] If $r(A)=r(A \cup p_1)=r(A\cup p_2)$, then $r(A\cup\{p_1,p_2\})=r(A).$ 
\end{enumerate}

If $M$ is a matroid on the ground set $S$, it is straightforward to define independent sets, spanning sets and bases directly from the rank function:
\begin{itemize}
  \item $I \subseteq S$ is {\it independent} if and only  if $r(I)=|I|$.
  \item $T \subseteq S$ is {\it spanning} if and only if $r(T)=r(S)$.
  \item $B \subseteq S$ is a {\it basis} if and only if $B$ is independent and spanning.
\end{itemize}
Thus, $B \subseteq S$ is a  basis of the matroid $M$ if $|B|=r(B)=r(S)$.

Three important operations  motivated by graph theory can be  defined for all matroids: duality, deletion, and contraction.  These are usually defined in terms of independent sets or bases, but it is possible to define all three operations via the rank function.


\begin{defn}\label{D:dual1} Let $M=(S,r)$ be a matroid.  Then define the dual matroid $M^*$  as follows:
$M^*=(S,r^*)$, where $r^*(A)=|A|+r(S-A)-r(S).$
\end{defn}

Using this definition, one can prove $r^*$ satisfies (R0), (R1) and (R2), and so defines a matroid.  One can also show that $B$ is a basis for $M^*$ if and only if $B=S-B'$, where $B'$ is a basis for $M$, i.e., the bases for $M^*$ are the complements of the bases of $M$.  (This is the way $M^*$ is usually  defined.)

We can also define deletion and contraction via the rank function and duality.

\begin{defn} \label{D:dc1} Let $M=(S,r)$ be a matroid, and let $p \in S$.  
\begin{enumerate}
\item \textit{Deletion}: $M-p=(S-p,r')$, where $r'(A)=r(A)$ for any $A \subseteq S-p$.
\item \textit{Contraction}: $M/p=(M^*-p)^*$.
\end{enumerate}

\end{defn}
Thus, both $M-p$ and  $M/p$ are matroids on the ground set $S-p$.  Again, one can prove this definition coincides with the (more familiar) definitions of deletion and contraction (in graphs or matroids) via independent sets or bases.  In particular, one can easily prove  $(M-p)^*=M^*/p$ and $(M/p)^*=M^*-p$.  Further, it is straightforward to show $r_{M/p}(A)=r(A \cup p)-r(p),$ where $r_{M/p}$ is the rank function of the contraction $M/p$ (and $r$ is the rank function of the original matroid $M$) -- see Theorem~\ref{T:contract-rank}.

\subsection{Generalized duality, deletion and contraction}
Brylawski observed that the definitions of the dual matroid (Definition~\ref{D:dual1}) and  deletion and contraction (Definition~\ref{D:dc1}) do not depend on the properties (R0), (R1) and (R2) that characterize the rank function of a matroid.  This leads to the next definition, Brylawski's generalized duality, deletion and contraction.

\begin{defn}\label{D:gens} Let $G=(S,r)$, where $S$ is a finite set and   $r:2^S \to \mathbb{Z}$ is a function satisfying $r(\emptyset)=0$.  Define duality, deletion and contraction:
\begin{itemize}
\item \textit{Duality} $G^*=(S,r^*)$, where $r^*(A):=|A|+r(S-A)-r(S)$.
\item \textit{Deletion} $G-p=(S-p,r')$, where $r'(A)=r(A)$ for all $A \subseteq S-p$.
\item \textit{Contraction} $G/p:=(G^*-p)^*$.
\end{itemize}
\end{defn}

This definition of duality has the usual involution property:  $(G^*)^*=G$.  We omit the routine proof.

\begin{prop}\label{P:G**}
Let $S$ be a finite set and $r:2^S \to \mathbb{Z}$ be any function  satisfying $r(\emptyset)=0$, where $G=(S,r)$.  Then $(G^*)^*=G$.
\end{prop}

The next result is useful for computing the rank of a subset in $G/p$, and will also be needed in our proof of a deletion-contraction recursion for the Tutte polynomial (Theorem~\ref{T:tutte-gen}(1)).

\begin{thm}[Brylawski] \label{T:contract-rank}  Let $G=(S,r)$, where $r:2^S\to \mathbb{Z}$  satisfies $r(\emptyset)=0$, and let $p \in S$.  Let $r_{G/p}$ be the rank function for $G/p$.  Then, for all $p \in S$ and $A \subseteq S-p$, 
$$r_{G/p}(A)=r(A \cup p)-r(p).$$
\end{thm}
\begin{proof}
Let $A \subseteq S-p$.  Then, in the dual $G^*$, we have $r^*(A)=|A|+r(S-A)-r(S)$.  This formula remains valid in $G^*-p$ (even though the sets $S-A$ and $S$ both contain $p$).  

Now let $r'$ be the rank function for $G^*-p$ and note that $r'$ is defined on the set $S-p$.  Then, for $p \notin A$, compute the rank function in $G/p=(G^*-p)^*$ as follows:
\begin{eqnarray*}
r_{G/p}(A) &=& |A|+r'((S-p)-A)-r'(S-p) \\
&=& |A|+\Big( |(S-p)-A|+r(S-(S-p-A))-r(S-p)\Big) - \\ 
&& \Big(|S-p| +r(S-(S-p))-r(S-p)\Big) \\
&=& r(A\cup p)-r(p).
\end{eqnarray*}

\end{proof}

\begin{exmp}\label{E:tree1}
Let $S=\{a,b,c\}$ and define the rank function as in Table~\ref{Ta:ex1-dual}.  This is the {\it branching greedoid} associated to the rooted tree of Figure~\ref{F:tree1}.  (Section~\ref{S:greedoid} gives more background information on greedoids.)  Then the rank of a subset of edges $A$ is the size of the largest rooted subtree contained in $A$.  For instance, we have $r(\{b,c\})=1$ because $c$ is the largest rooted subtree contained in $\{b,c\}$.

Then $r$ is {\it not} the rank function of a matroid because, for example, $r(\{b,c\})=1$ while $r(S)=3$, so the unit rank increase matroid property (R1) is violated.   We use Definition~\ref{D:dual1} to find the rank for the dual $G^*$ -- see the last row of Table~\ref{Ta:ex1-dual}.

\begin{table}[h]
\begin{center}
\begin{tabular}{|c||c|c|c|c|c|c|c|c|} \hline
$A$ &$\emptyset$ & $a$ & $b$ &  $c$ & $ab$ &  $ac$ & $bc$ &  $abc$  \\ \hline
$r(A)$ &$0$ & $1$ & $0$ &  $1$ & $2$ &  $2$ & $1$ &  $3$  \\ \hline
$r^*(A)$  &$0$ & $-1$ & $0$ &  $0$ & $0$ &  $-1$ & $0$ &  $0$  \\ \hline
\end{tabular}
\end{center}
\medskip
\caption{Rank function for rooted tree $G$ of Figure~\ref{F:tree1} and its generalized dual $G^*$.} \label{Ta:ex1-dual}
\end{table}%

\vspace{-.3in}
\begin{figure}[h]
\centerline{\includegraphics[width=.8in]{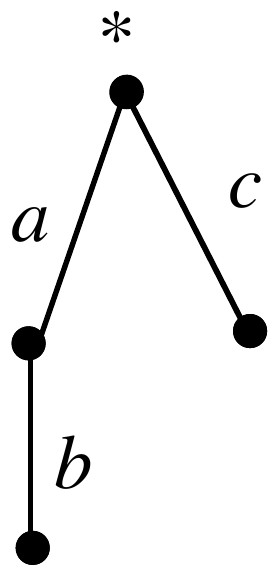}}
\caption{Rooted tree $G$ for Example~\ref{E:tree1}.}
\label{F:tree1}
\end{figure}

For instance, the dual rank $r^*(\{a,c\})=|\{a,c\}|+r(b)-r(S)=2+0-3=-1$.  This gives another way to see that $r$ is not the rank function of a matroid, since, if it were, the dual $G^*$ would also be a matroid.  But $r(A)<0$ is impossible for matroids.

For the deletion $G-a$, we simply compute the rank function by restricting $r$ to subsets avoiding $a$.  The rank in the contraction can be computed via duality (as in Definition~\ref{D:dc1}), or directly from Theorem~\ref{T:contract-rank}.  See Table~\ref{Ta:ex1-dc}.

\begin{table}[h]
\begin{center}
\begin{tabular}{|c||c|c|c|c|} \hline
$A$ &$\emptyset$ & $b$ &  $c$ &   $bc$   \\ \hline
rank in $G-a$ &$0$ & $0$ & $1$ &  $1$   \\ \hline
rank in $G/a$  &$0$ & $1$ & $1$ &  $2$   \\ \hline
\end{tabular}
\end{center}
\medskip
\caption{Rank function for $G-a$ and $G/a$ for the rooted tree of Figure~\ref{F:tree1}.} \label{Ta:ex1-dc}
\end{table}%

For rooted graphs,  one can check that these definitions of deletion and contraction via the rank function correspond  to the usual graph theoretic operations of deletion and contraction of edges. 

\end{exmp}

It is straightforward to generalize direct sums to arbitrary rank functions.

\begin{defn}\label{D:dirsum} Let $G_1=(S_1,r_1)$ and $G_2=(S_2,r_2)$ where $S_1$ and $S_2$ are disjoint sets.  For $A_i \subseteq S_i$, define $r(A_1\cup A_2)=r_1(A_1)+r_2(A_2)$.  Then $G_1 \oplus G_2=(S_1\cup S_2,r)$ is the {\it direct sum} of $G_1$ and $G_2$.

\end{defn}
We omit the immediate proof of the next proposition.

\begin{prop}\label{P:dirsum}
Let $G_1\oplus G_2$ be the direct sum of $G_1$ and $G_2$.  Then $$(G_1\oplus G_2)^*=G_1^*\oplus G_2^*.$$
\end{prop}
\section{The Tutte polynomial}\label{S:tutte}
The Tutte polynomial is an important two-variable invariant for graphs and matroids.  An extensive introduction to this polynomial can be found in \cite{bo}.  We now extend the definition of the Tutte polynomial by using an arbitrary rank function $r:2^S \to \mathbb{Z}$.

\begin{defn}\label{D:gentutte}
Let $r:2^S \to \mathbb{Z}$ be any function.  Then define a  function $f(G;t,z)$ for $G=(S,r)$:
$$f(G;t,z)=\sum_{A \subseteq S}t^{r(S)-r(A)}z^{|A|-r(A)}.$$
\end{defn}
Definition~\ref{D:gentutte} generalizes the Tutte polynomial of a matroid to an arbitrary rank function.  We do not assume any special properties for the  function $r$.  The exponent $r(S)-r(A)$ is the {\it corank} of $A$, and $|A|-r(A)$ is the {\it nullity} of $A$.

\begin{exmp}
Let $G=(S,r)$ for $S=\{a,b\}$ with the rank function $r:2^S \to \mathbb{Z}$  given as follows:  $r(\emptyset)=3, r(a)=-1, r(b)=7$ and $r(S)=2$.  Then $$f(G;t,z)=\frac{1}{t^5 z^6}+\frac{1}{t z^3}+t+1.$$
Thus, $f(S;t,z)$ need not be a polynomial.
\end{exmp}

We will generally assume $r(\emptyset)=0$; this ensures $r^*(\emptyset)=0$, but is also needed to prove $G^{**}=G$ (Proposition~\ref{P:G**}).  Further, writing $f(G;t,z)=\sum_{A \subseteq S} t^kz^m$, it is evident that 
\begin{itemize}
\item $k \geq 0$ for all $A \subseteq S$ if and only if $r(A) \leq r(S)$ for all $A \subseteq S$ (rank $S$ maximum), and
\item $m \geq 0$  for all $A \subseteq S$ if and only if $r(A) \leq |A|$ for all $A \subseteq S$ (subcardinal).
\end{itemize}
Thus, $f(G;t,z)$ will be a polynomial precisely when the rank function satisfies these two properties.

Applying Definition~\ref{D:gens} to this polynomial, we  get the following.

\begin{thm} [Brylawski] \label{T:tutte-gen} Let $r:2^S \to \mathbb{Z}$ be any function satisfying $r(\emptyset)=0$.  Then 
\begin{enumerate}
  \item Deletion-contraction:  For any $p \in S$, $$f(G;t,z)=t^{r(G)-r(G-p)}f(G-p;t,z)+z^{1-r(p)}f(G/p;t,z),$$
  where $G-p$ and $G/p$ are defined in Definition~\ref{D:gens}.
  \item Duality:  Let $G^*$ be the dual of $G$ in the sense of Definition~\ref{D:gens}.  Then $$f(G^*;t,z)=f(G;z,t).$$ 
\end{enumerate}

\end{thm}
\begin{proof}
(1) 
Let  $r'$ be  the rank function in $G-p$ and $r''$ the rank function in $G/p$.  Then,  
  \begin{itemize}
  \item Deletion:  $r'(A)=r(A)$ for all $A \subseteq S-p$.  
  \item Contraction:   If $p \in A$, then $r''(A-p)=r(A)-r(p)$.  (This follows from Theorem~\ref{T:contract-rank}.)
\end{itemize}
 We break up the subsets of $S$ into two classes:  Let $\mathcal{S}_1$ be the collection of all subsets of $S$ containing $p$, and $\mathcal{S}_2$ be the collection of all subsets of $S$ avoiding $p$. 

  Then 
$$ f(G;t,z)=\sum_{A \in \mathcal{S}_1} t^{r(S)-r(A)} z^{|A|-r(A)}  +  \sum_{A \in \mathcal{S}_2} t^{r(S)-r(A)} z^{|A|-r(A)}.$$

{\sc Case 1}:   $A \in \mathcal{S}_1$.  Then  $r''(S-p)=r(S)-r(p)$ and $r''(A-p)=r(A)-r(p)$, so the corank of $A$ (computed in $G$) equals the corank of $A-p$ (computed in $G/p$):  $$r(S)-r(A)=r''(S-p)-r''(A-p).$$

For the nullity,  we have    $$|A|-r(A)=|A-p|+1-r''(A-p)-r(p).$$
Thus 
\begin{eqnarray*}
\sum_{A \in \mathcal{S}_1} t^{r(S)-r(A)} z^{|A|-r(A)} &=& \sum_{A  \in \mathcal{S}_1} t^{r''(S-p)-r''(A-p)} z^{(|A-p|-r''(A-p))+(1-r(p))} \\
&=& z^{1-r(p)} \sum_{B \subseteq S-p} t^{r''(S-p)-r''(B)} z^{|B|-r''(B)} \mbox{ (for }B=A-p) \\
&=&z^{1-r(p)}f(G/p;t,z).
\end{eqnarray*}

{\sc Case 2}:   $A \in \mathcal{S}_2$.  Then $p \notin A, r(S)=r'(S-p)$ and $r(A)=r'(A)$.  We compute the corank of $A$ in both $G$ and $G-p$:
\begin{eqnarray*}
r(S)-r(A)&=&r(S)-r(S-p)+r(S-p)-r(A) \\
&=&(r(S)-r(S-p))+(r'(S-p)-r'(A)).
\end{eqnarray*}

For the nullity, there is no change this time:
$|A|-r(A)=|A|-r'(A)$.  Thus
\begin{eqnarray*}
\sum_{A \in \mathcal{S}_2} t^{r(S)-r(A)} z^{|A|-r(A)} &=&t^{r(S)-r(S-p)} \sum_{A  \in \mathcal{S}_2} t^{r'(S-p)-r'(A)} z^{|A|-r'(A)} \\
&=&t^{r(S)-r(S-p)}f(G-p;t,z).
\end{eqnarray*}
Combining these two cases gives us our deletion-contraction recursion.

(2)  Note that, for any $A \subseteq S$, we have  $r^*(S)-r^*(A)=|S-A|-r(S-A)$ (since $r(\emptyset)=0$) and  $|A|-r^*(A)=r(S)-r(S-A)$.  The result then follows from Definition~\ref{D:gens}(1).

\end{proof}

The deletion-contraction recursion of Theorem~\ref{T:tutte-gen}(1) is a generalization of a greedoid version of this formula that appears as Proposition 2.5 in  \cite{gm}.  In that formula, the $z^{1-r(p)}$ coefficient of the contraction term does not appear since $r(p)=1$ for all points $p$ that we contract.  See Section~\ref{S:greedoid} below.

\begin{exmp}\label{E:tree2}  Returning to Example~\ref{E:tree1}, we first compute $f(G;t,z)$:

\begin{table}[h]
\begin{center}
\begin{tabular}{|c||c|c|c|c|c|c|c|c|} \hline
$A$ &$\emptyset$ & $a$ & $b$ &  $c$ & $ab$ &  $ac$ & $bc$ &  $abc$  \\ \hline
$r(A)$ &$0$ & $1$ & $0$ &  $1$ & $2$ &  $2$ & $1$ &  $3$  \\ \hline
Term &$t^3$ & $t^2$ & $t^3z$ &  $t^2$ & $t$ &  $t$ & $t^2z$ &  $1$   \\ \hline
\end{tabular}
\end{center}
\end{table}%

Then 
\begin{eqnarray*}f(G;t,z)&=&(t+1)(t^2 z+t^2+t+1),\\
 f(G-a;t,z)&=&1+t+z+tz, \\
 f(G/a;t,z)&=&(t+1)^2.
\end{eqnarray*}

In this case, $r(a)=1$ and $r(S)-r(S-a)=2$, and the reader can verify $f(G;t,z)=f(G/a;t,z)+t^2f(G-p;t,z),$ as required by \ref{T:tutte-gen}(1).

If, instead, we delete and contract $b$,  we find $f(G-b;t,z)=(t+1)^2$ and $f(G/b;t,z)=t^3+t^2+\frac{t}{z}+\frac{1}{z}$ (and so $f(G/b)$ is not a polynomial).  Now $r(b)=0$ and $r(S)-r(S-b)=1$, so \ref{T:tutte-gen} gives $f(G;t,z)=z f(G/b)+t f(G;t,z)$, which the reader can again verify.

For the dual $G^*$, we find:
\begin{table}[h]
\begin{center}
\begin{tabular}[b]{|c||c|c|c|c|c|c|c|c|} \hline
Subset &$\emptyset$ & $a$ & $b$ &  $c$ & $ab$ &  $ac$ & $bc$ &  $abc$  \\ \hline
Dual rank &$0$ & $-1$ & $0$ &  $0$ & $0$ &  $-1$ & $0$ &  $0$  \\ \hline
Dual term &$1$ & $tz^2$ & $z$ &  $z$ & $z^2$ &  $tz^3$ & $z^2$ &  $z^3$  \\ \hline
\end{tabular}
\end{center}
\end{table}

Thus $f(G^*;t,z)=f(G;z,t)$, and we remark that a term $t^mz^n$ of $f(G)$ corresponding to a subset $A$ gives rise to the term $t^nz^m$ in $f(G^*)$ corresponding to the subset $S-A$.

\end{exmp}

We remark that when $G$ is a matroid, then $r(G)=r(G-p)$ (provided $G$ is not an isthmus) and $r(p)=1$ (provided $p$ is not a loop).  Thus the recursion of Theorem~\ref{T:tutte-gen}(1) reduces to the familiar $f(G;t,z)=f(G-p;t,z)+f(G/p;t,z)$.  It is also possible to interpret this recursion for isthmuses and loops;  see \cite{gt} for one approach.

\section{Greedoids}\label{S:greedoid}
Greedoids are a generalization of matroids that were first introduced in \cite{kl}.  An extensive introduction appears in \cite{bz}.  Although there are fewer axiomatizations of greedoids than there are of matroids, it is still possible define greedoids from a  rank function.

\begin{defn}\label{D:greedoid} A {\it greedoid} $G$ is a pair $(S,r)$ where $S$ is a finite set and $r:2^S \to \mathbb{Z}^+\cup\{0\}$ such that:
\begin{enumerate}
\item [(Gr0)] $r(\emptyset)=0$ [normalization]
\item [(Gr1)] $r(A) \leq r(A \cup \{p\})$ [increasing]
\item [(Gr2)] $r(A) \leq |A|$ [subcardinal]
\item [(Gr3)] If $r(A)=r(A \cup p_1)=r(A\cup p_2)$, then $r(A\cup\{p_1,p_2\})=r(A).$ [local semimodularity]
\end{enumerate}
\end{defn}

Note that matroids satisfy these four properties, so matroids are greedoids.  We also remark that the greedoid normalization axiom (Gr0) is the same as the matroid axiom (R0), and the local semi-modularity for greedoids (Gr3) is identical to the matroid version (R2$'$).

When $r(A)=|A|$, we call $A$ a {\it feasible set} in the greedoid.  Thus, feasible sets in greedoids play the same role as independent sets in matroids.  {\it Bases} are defined to be maximal feasible sets, and, as with matroids, all bases have the same cardinality.  One important difference between matroids and greedoids is that a greedoid  is {\it not} uniquely determined by its collection of bases.  In fact, there are, in general, many non-isomorphic greedoids in which $S$ is a basis.  (A greedoid $G$ is {\it full} if $S$ is a basis of $G$, i.e., $r(G)=|S|$.)

As an example, consider the rooted tree of Example~\ref{E:tree1}.  This is a greedoid; more generally, if $G$ is a rooted graph, i.e., a graph with a distinguished vertex,  with edges $S$, we get a greedoid on the ground set $S$ by defining the feasible sets to be the rooted subtrees of $G$.  This is the {\it branching greedoid} associated to the rooted graph $G$.  Note that rooted trees are full greedoids in this context.

Deletion and contraction in greedoids are usually defined in terms of feasible sets.  

\begin{defn}\label{D:dc-greed}
Let $G$ be a greedoid on the ground set $S$.  For $p \in S$, define the feasible sets of the deletion $G-p$ and contraction $G/p$ as follows:
\begin{enumerate}
  \item Deletion:  $F \subseteq S-p$ is feasible in  $G-p$ if $F$ is feasible in $G$.
    \item Contraction:  $F\subseteq S-p$  is feasible in  $G/p$ if $F\cup p$ is feasible in $G$.
    \end{enumerate}
\end{defn}

Using Definition~\ref{D:dc-greed}, one can show $G-p$ is always a greedoid, but $G/p$ is a greedoid  if and only if $\{p\}$ is a feasible set (or $p$ is a greedoid loop).  

\begin{prop}\label{P:bad-contract}  Let $G=(S,r)$ be a greedoid and suppose $p \in S$ is in some feasible set $F$.  Then $G/p$ is a greedoid if and only if $p$ is feasible.
\end{prop}
\begin{proof}
Suppose $\{p\}$ is feasible.  Then $G/p$ is a greedoid by Proposition~\ref{P:dc-greed}(2) (below).  It remains to show that $G/p$ is not a greedoid when $p$ is not feasible.  Let $F$ be any feasible set containing $p$.  Then, by Theorem~\ref{T:contract-rank}, $r_{G/p}(F-p)=r(F)-r(p)=|F|$.  But this violates the subcardinal property (Gr2) of greedoid rank functions.

\end{proof}

If $p$ is a greedoid loop, it is easy to  check that Definition~\ref{D:dc1} gives $G-p=G/p$.  Definitions~\ref{D:dc1} (based on the rank function) and \ref{D:dc-greed} (based on feasible sets) agree for deletion and contraction in greedoids (where we assume  $p$ is a feasible singleton when defining contraction using \ref{D:dc-greed}).  We omit the straightforward proof.

\begin{prop}\label{P:dc-greed} Let $G=(S,r)$ be a greedoid and let $p \in S$.  Then 
\begin{enumerate}
  \item $G-p$ has rank function $r|_{S-p}$.
  \item If  $p \in S$ is feasible, then $G/p$ has rank function $r_{G/p}$ satisfying $r_{G/p}(A)=r(A\cup p)-r(p).$
\end{enumerate}

\end{prop}

If $p$ is in no feasible sets in the greedoid $G$, we call $p$ a {\it greedoid loop}.  Then, by Theorem~\ref{T:contract-rank}, we see $G/p=G-p$ in this case.  

We can formulate dual versions of  the greedoid rank axioms (Gr0) -- (Gr3).  This gives us a direct characterization of duality for these structures.

\begin{thm} [Brylawski]  \label{T:dual-greedoid1}  Let $r$ be the rank function for a greedoid $G=(S,r)$.  Then, for all $B \subseteq S$, the dual rank $r^*$ satisfies the following:

\begin{enumerate}
\item [(Gr0$^*$)] $r^*(\emptyset)=0$ [normalization]
\item [(Gr1$^*$)] $r^*(B\cup p) \leq r^*(B)+1$ [unit rank increase]
\item [(Gr2$^*$)] $r^*(B) \leq r^*(S)$ [rank $S$ maximum]
\item [(Gr3$^*$)] If $r^*(B-p)=r^*(B - q)=r^*(B)-1$, then $r^*(B-\{p,q\})=r^*(B)-2.$  [local rank decrease]
\end{enumerate}

\end{thm}
\begin{proof}
We omit the straightforward proofs of (Gr0$^*$) and (Gr2$^*$).  For (Gr1$^*$), assume $p \notin B$ and set $A=S-(B\cup p)$, so $A \cup p = S-B$.  

Then 
\begin{eqnarray*}
r^*(B\cup p) &=& |B \cup p| + r(S-(B \cup p)) -r(S) \\
&=& |B|+1+r(A)-r(S) \\
& \leq & |B|+1+r(A \cup p) -r(S) \mbox{   (By (Gr1): } r(A)\leq r(A \cup p)) \\
&=& |B|+1+r(S-B) -r(S) \\
&=& r^*(B)+1
\end{eqnarray*}

For (Gr3$^*$), set $A=S-B$.  Then $r^*(B)=r^*(B-p)+1$ implies $r(A)=r(A\cup p)$, and, similarly, $r(A)=r(A \cup q)$.  By (Gr3), we then get $r(A\cup \{p,q\})=r(A)$.  Then 
\begin{eqnarray*}
r^*(B-\{p,q\})&=&|B-\{p,q\}|+r(A\cup\{p,q\})-r(S) \\
&=&|B|-2+r(A)-r(S) \\
&=& |B|+r(S-B)-r(S)-2 \\
&=&r^*(B)-2.
\end{eqnarray*}

\end{proof}

In Example~\ref{E:tree1}, the rooted tree is a greedoid, but its dual is not (in particular, the rank function $r^*$ for $G^*$ was negative for some subsets).  When is the dual of a greedoid also a greedoid?  The answer leads us back to matroids.

\begin{thm} [Brylawski] \label{T:dual-greedoid2}
Let $\mathcal{M}$ be the class of all matroids, $\mathcal{G}$ the class of all greedoids and $\mathcal{G}^*=\{G^*:G \in \mathcal{G}\}$.  Then
$$\mathcal{G} \cap \mathcal{G}^* = \mathcal{M}.$$
\end{thm}

\begin{proof}
Matroids are greedoids, so $\mathcal{M}\subseteq \mathcal{G}$. Taking the duals gives $\mathcal{M}^*  \subseteq \mathcal{G}^*$.  Since duals of matroids are matroids, we also have $\mathcal{M}^* = \mathcal{M}$, so $\mathcal{M}\subseteq \mathcal{G}\cap \mathcal{G}^*$.

For the converse, note that if $G \in  \mathcal{G}\cap \mathcal{G}^*$, then the rank function $r$ for $G$ satisfies (Gr0), (Gr1$^*$) and (Gr3).  But these three properties characterize the rank function of a matroid, so $G$ is a matroid, i.e., $\mathcal{G}\cap \mathcal{G}^* \subseteq \mathcal{M}$.

\end{proof}

When $G$ is a rooted graph, recall that a subset of edges $A$ is a  feasible set in the branching greedoid if the edges of $A$ form a rooted tree.  Example~\ref{E:tree1} shows the dual $G^*$ is not generally a greedoid.  But we can determine precisely when $r^*(A)\geq 0$ for all $A \subseteq S$ in this case.

\begin{prop}\label{P:dual-rank-rooted-graph}
Let $G$ be a connected rooted graph with edges $S$ and branching greedoid rank function $r$.  Then $r^*(A) \geq 0$ for all $A \subseteq S$ if and only if every vertex of $G$ is adjacent to the root.
\end{prop}
\proof
Let  $V=\{v_0, v_1, v_2, \dots, v_n\}$ be the collection of vertices of $G$, with root vertex $v_0$, and suppose $v_0$ is adjacent to each $v_i$ for $1 \leq i \leq n$.  Write $e_i$ for the edge joining vertices $v_0$ and $v_i$.  Let $A\subseteq S$ and suppose $e_i \in A$ for $1 \leq i \leq k$ and $e_j \notin A$ for $k+1 \leq j \leq n$ (for some $k$).  Then $r(S)=n, |A|\geq k$ and $r(S-A)\geq n-k$, so $$r^*(A)=|A|+r(S-A)-r(S)\geq k+(n-k)-n \geq 0.$$

For the converse, suppose the root vertex $v_0$ is adjacent to vertices $v_1, v_2, \dots, v_n$, but there is some vertex $u$ that  is not adjacent to $v_0$.  Let $e_i$ be the edge joining $v_0$ and $v_i$ (as above), and set $A=\{e_1,e_2,\dots, e_n\}$.  Then $|A|=n$ and $r(S-A)=0$.  Since $G$ is connected and has at least $n+1$ non-root vertices, we must have $r(S)>n$.  Thus, $r^*(A)=n+0-r(S)<0.$
\qed

\section{Antimatroids, demi-matroids  and duality}\label{S:matroid-gen}
\subsection{Antimatroids} Antimatroids are an important class of greedoids that have been rediscovered many times in the literature; they   originally appeared in \cite{d} in 1940.  An interesting account of the history of the different formulations and discoveries of antimatroids appears in \cite{m}.  For our purposes, an antimatroid is a greedoid in which the union of feasible sets is always feasible.

\begin{defn}\label{D:anti}  Let $G=(S,r)$ be a greedoid with rank function $r$ and feasible sets $\mathcal{F}$.  Then $G$ is an {\it antimatroid} if $F_1 \cup F_2 \in \mathcal{F}$ whenever $F_1, F_2 \in \mathcal{F}$.
\end{defn}

There are several important combinatorial structures that admit an antimatroid structure in a natural way: a partial list includes trees, rooted trees, rooted directed trees, finite subsets of Euclidean space, posets (in three different ways), and vertices in chordal graphs.  If the antimatroid has no greedoid {\it loops} (elements in no feasible set), then the antimatroid is a full greedoid, i.e., $S$ is a feasible set.  All of the antimatroids listed above are full.

We are interested in interpreting the dual rank function of an antimatroid directly from the antimatroid structure.  One obvious problem is the presence of subsets with negative rank.  In fact, in Example~\ref{E:tree1}, the dual rank  $r^*(A)\leq 0$ for all $A\subseteq S$ (see Table~\ref{Ta:ex1-dual}).  This is true generally when the greedoid is full.

\begin{prop}\label{P:neg-rank}
Let $G=(S,r)$ be a full greedoid.  Then $r^*(A)\leq 0$ for all $A \subseteq S$.
\end{prop}
\proof
$G$ is a full greedoid if and only if $r(S)=|S|$.  Then $r^*(A)=|A|+r(S-A)-r(S)=r(S-A)-|S-A|\leq 0$ by the subcardinal greedoid rank property (Gr2).

\qed

When $G$ is an antimatroid with $S$ feasible, we say a subset $C \subseteq S$ is {\it convex} if $S-A$ is feasible.  The convex sets provide a complementary way to study antimatroids, and arise naturally in a variety of combinatorial structures.  See Example~\ref{E:tree2}.

\begin{defn}\label{D:convex-closure}  Let $(S,r)$ be an antimatroid and let $A \subseteq S$.  The {\it convex closure} $\overline{A}$ is the smallest convex set that contains $A$.
\end{defn}

Equivalently, $\overline{A}$ is the intersection of all convex sets containing $A$.

\begin{exmp}\label{E:tree2}  Let $T$ be the (non-rooted) tree  in Figure~\ref{F:tree2}.  Let $S$ be the set of edges of the tree, and define $A \subseteq S$ to be feasible if $S-A$ is a subtree.  Then this gives an antimatroid structure, where the rank function $r(A)$ is the size of the largest feasible subset of $A$.  Then $C$ is convex if and only if the edges of $C$ form a subtree.

\begin{figure}[h]
\centerline{\includegraphics[width=2.75in]{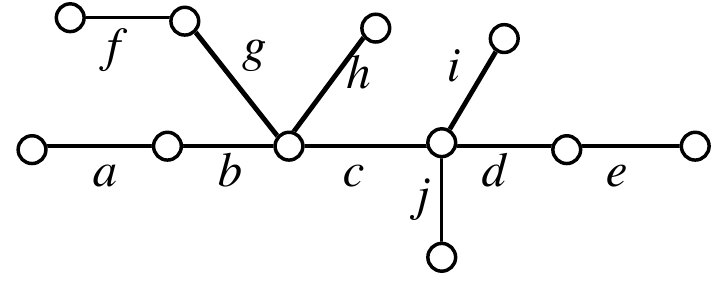}}
\caption{Tree for Example~\ref{E:tree2}.}
\label{F:tree2}
\end{figure}

This is the {\it pruning} antimatroid associated to the tree.  For instance, the subset $A=\{a,d,e,f\}$ is feasible since its complement $S-A=\{b,c,g,h,i,j\}$ is a subtree.  Thus, $S-A$ is convex.  Equivalently, $A$ is feasible if the edges of $A$ can be {\it pruned} from the tree by repeatedly removing leaves.  Since $A$ is feasible, we have $r(A)=4$, and $r^*(S-A)=|S-A|+r(A)-r(S)=0.$  

For $A=\{b,e,h\}$, we find $A$ is not feasible.  Then $r(A)=2$ since the subset $\{e,h\}$ is the largest feasible subset of $A$,    We also see that $A$ is not convex.  The smallest subtree containing $A$ is $\{b,c,d,e,h\}$, so $\overline{A}=\{b,c,d,e,h\}$. 

\end{exmp}

We now give a combinatorial interpretation for the dual rank in an antimatroid.

\begin{thm}\label{T:anti-rank}  Let $G=(S,r)$ be a full antimatroid.  Then $r^*(A)=-|\overline{A}-A|$, where $\overline{A}$ is the convex closure of $A$.

\end{thm}
\proof  We first show $r(S-A)=r(S-\overline{A})$.  Since $S-\overline{A} \subseteq S-A$, we know $r(S-\overline{A}) \leq r(S-A)$ (from greedoid rank property (Gr1)).  Since $\overline{A}$ is convex, we know $S-\overline{A}$ is feasible, so $r(S-\overline{A})=|S-\overline{A}|.$

Now suppose $F \subseteq (S-A)$ is feasible with $x \in F \cap \overline{A}$.  Then we can find a feasible set $F'$ such that $F' \subseteq S-\overline{A}$ and $F' \cup x$ is feasible.  (This follows from the fact that all feasible sets are {\it accessible} -- simply remove elements from $F$, one by one, maintaining feasibility.)

But then $F' \cup x$ and $S-\overline{A}$ are both feasible, so $F'\cup x \cup (S-\overline{A})=(S-\overline{A})\cup x$ must be feasible.   But this gives $\overline{A}-x$ convex, with $A \subseteq \overline{A}-x$.  Thus, $\overline{A}$ is not the smallest convex set containing $A$, a contradiction.  We conclude that every feasible subset $F$ of $S-A$ is a subset of $S-\overline{A}$, so $r(S-\overline{A})=r(S-A)$.

Now 
\begin{eqnarray*}
r^*(A)&=&|A|+r(S-A)-r(S) \\
&=& |A|+r(S-\overline{A})-r(S) \\
&=& |A|+|S-\overline{A}|-|S| \\
&=& -|\overline{A}-A|.
\end{eqnarray*}

\qed

For example, let $A=\{a,d,f\}$ in the tree in Figure~\ref{F:tree2}.  Then $r(S-A)=4$ since $\{e,h,i,j\}$ is the largest feasible subset of $S-A$, so $r^*(A)=|A|+r(S-A)-r(S)=3+4-10=-3$.  

Now $\overline{A}$ is the smallest subtree containing $A$, so $\overline{A}=\{a,b,c,d,f,g\}$.  Thus $\overline{A}-A=\{b,c,g\}$, and $|\overline{A}-A|=3$, as required by Theorem~\ref{T:anti-rank}.

An immediate corollary of Theorem~\ref{T:anti-rank} is the following characterization of convex sets in an antimatroid in terms of the dual rank $r^*$.

\begin{cor}
Let $G=(S,r)$ be a full antimatroid. Then $C$ is convex if and only if $r^*(C)=0$.
\end{cor}

\subsection{Demi-matroids}Demi-matroids were introduced in \cite{bjms}, where they provide a more general setting for Wei's duality theorem for codes \cite{w}. 

\begin{defn} \label{D:demi} A {\it demi-matroid} is a triple $(S,r,s)$ with $S$ a finite set and rank functions $r,s:2^S\to \mathbb{Z}^+\cup\{0\}$ satisfying
\begin{enumerate}
  \item $r(A) \leq |A|$ and $s(A) \leq |A|$
  \item If $A \subseteq B$, then $r(A) \leq r(B)$ and $s(A) \leq s(B)$.
  \item $|S-A|-r(S-A)=s(S)-s(A)$.
\end{enumerate} 

\end{defn}

It follows immediately from this definition that $r$ and $s$ also satisfy the complementary version of (3):
$$|S-A|-s(S-A)=r(S)-r(A).$$

If $M$ is a matroid with rank function $r$, then the function $s$ is simply the dual rank $r^*$.  Thus, matroids are demi-matroids, where $s=r^*$.  However, if $S$ is a finite set with arbitrary rank function $r$, then the generalized dual rank function $r^*$ of Definition~\ref{D:gens} need not satisfy the properties required of $s$.

For instance, consider the greedoid of Example~\ref{E:tree1}.  Then $r(A)$ is a non-negative integer for any subset $A$, and the greedoid rank function satisfies properties (1) and (2) of the demi-matroid properties (Definition~\ref{D:demi}).  Further, $r$ and the dual rank function $r^*$ satisfy (3):
$$|S-A|-r(S-A)=r^*(S)-r^*(A).$$
However, the dual rank function $r^*$ does not satisfy  (2) and $r^*(A)<0$ for $A=\{a\}$.  (Definition~\ref{D:demi}(2) is violated for $A= \{c\}$ and $B=\{a,c\}$.)

 We are interested in characterizing demi-matroids via the rank function $r$.  We will need the following lemma, whose straightforward inductive proof is omitted.

\begin{lem}\label{L:ranklemma}  Let $S$ be a finite set with rank function $r:2^S \to \mathbb{Z}^+\cup\{0\}$ satisfying the subcardinal property:  $r(A) \leq r(B)$ whenever $A \subseteq B$.  Then the following two properties are equivalent:
\begin{enumerate}
\item [(R1)] $r(A) \leq r(A \cup p) \leq r(A)+1$ [unit rank increase]
  \item[(MN)] If $A \subseteq B$, then  $|A|-r(A) \leq |B|-r(B)$. [monotone nullity]
\end{enumerate}
\end{lem}
We also point out that monotone nullity (propertry (MN)) can be expressed in other ways.  For instance, it is immediate that this property is equivalent to 
$$A \subseteq B \Rightarrow r(B)-r(A)\leq |B-A|.$$

\begin{thm}\label{T:demi}  Let $S$ be a finite set with rank function $r:2^S\to \mathbb{Z}$.  Then the triple $(S,r,r^*)$ is a demi-matroid if and only if $r$ satisfies, for all $p \in S$ and  $A, B \subseteq S$:
\begin{enumerate}
  \item [(a)] $ 0 \leq r(A) \leq |A|$, [nonnegative, subcardinal]
  \item[(b)] if $A\subseteq B$, then $r(A) \leq r(B)$, [monotone rank]
  \item[(c)] $r(A \cup p) \leq r(A)+1$ [unit rank increase]
\end{enumerate}

\end{thm}
\proof  Suppose $r$ satisfies the three conditions (a), (b) and (c).  We show the triple $(S,r,r^*)$ is a demi-matroid.  From Definition~\ref{D:gens}, we have $r^*(A)=r(S-A)+|A|-r(S)$.  This immediately  implies that (3) holds in Definition~\ref{D:demi}:  $|S-A|-r(S-A)=r^*(S)-r^*(A)$ for any $A \subseteq S$.  Thus, to show $(S,r,r^*)$ is a demi-matroid, we must show the dual rank $r^*$ also satisfies (for all subsets $A, B \subseteq S$):
\begin{enumerate}
  \item $0\leq r^*(A)\leq |A|$, and
  \item if $A \subseteq B$, then $r^*(A)\leq r^*(B)$.
\end{enumerate}

By Lemma~\ref{L:ranklemma}, we may assume $r$ satisfies the monotone nullity property (MN); it will be easier to use this property in the proof.  First, we show $0\leq r^*(A)$ for all $A$.  But  $0 \leq r^*(A)$ if and only if $0\leq r(S-A)+|A|-r(S)$, i.e., $r(S)\leq r(S-A)+|A|$.  This now follows from monotone nullity (Lemma~\ref{L:ranklemma}):  setting $A'=S-A$ and $B'=S$, we have $A' \subseteq B'$, so we have  $|A'|-r(A') \leq |B'|-r(B')$.  Rewriting:  $|S-A|-r(S-A) \leq |S|-r(S)$, which is clearly equivalent to $r(S)\leq r(S-A)+|A|$.

To show $r^*(A) \leq |A|$, we note this is equivalent to  $r(S-A)+|A|-r(S) \leq |A|$, i.e., $r(S-A)\leq r(S)$.  This now follows directly from condition (b).

It remains to show that if $A \subseteq B$, then $r^*(A) \leq r^*(B)$.  Now 
\begin{eqnarray*}
r^*(A) \leq r^*(B) & \Leftrightarrow& r(S-A)+|A|-r(S) \leq r(S-B)+|B|-r(S)\\
& \Leftrightarrow& r(S-A)-r(S-B) \leq |S-A|-|S-B| \\
& \Leftrightarrow& |B'|-r(B') \leq |A'|-r(A')
\end{eqnarray*}
where $A'=S-A$ and $B'=S-B$, with $B' \subseteq A'$.  This now follows from monotone nullity (Lemma~\ref{L:ranklemma}).

For the converse, we first observe that if $(S,r,s)$ is a demi-matroid, then (3) in Definition~\ref{D:demi} forces $s=r^*$, where $r^*(A)=r(S-A)+|A|-r(S)$.  Then  $r$ must satisfy conditions (a) and (b) (this follows from Def~\ref{D:demi}(1) and (2)).

It remains to show $r$ also satisfies the monotone nullity property (MN) (Lemma~\ref{L:ranklemma}).  Assume $A \subseteq B$.  Then the argument given above shows that 
$r$ satisfies (MN) if and only if $r^*$ satisfies Def~\ref{D:demi}(2):  $$|A|-r(A) \leq |B|-r(B) \Leftrightarrow r^*(S-B)\leq r^*(S-A).$$  Since $S-B \subseteq S-A$ and $(S,r,r^*)$ is a demi-matroid, we know $r^*(S-B)\leq r^*(S-A)$.  This completes the proof.

\qed

In Example~\ref{E:tree1}, the greedoid rank function is nonnegative and subcardinal, so it satisfies condition (a) of Theorem~\ref{T:demi}.  Further, $r$ satisfies the monotone rank property (condition (b)).  But $r$ does not satisfy the unit rank increase property (R1).  (Set $A=\{b\}$ and $B=\{a,b\}$, for instance.)    This confirms the fact that the greedoid of Example~\ref{E:tree1} is not a demi-matroid.  (Of course, $r$ does not satisfy the monotone nullity property, either, which is violated for the same $A$ and $B$.)

Lemma 1 of ~\cite{bjms} shows that the rank function $r$ of a demi-matroid satisfies the unit rank increase property (R1).   It's also worth pointing out that the three properties (a), (b) and (c) of Theorem~\ref{T:demi} are not sufficient to define a matroid.  For instance, Example 2 of \cite{bjms} has $S=\{a,b\}$, with $r(A)=0$ for $A=\emptyset, \{a\}$ or $\{b\}$, and $r(S)=1$. Then $r^*=r$, and  $(S,r,r^*)$ is a demi-matroid, so $r$ and $r^*$ satisfy  (a), (b) and (c) from Theorem~\ref{T:demi}.  But this is not a matroid; the semimodular property (R2) is violated.

\end{document}